\documentclass{amsart}

\usepackage{amssymb}
\newcommand{\cz}[2]{C_{#1}(#2)}

\newcommand{\n}[2]{N_{#1}(#2)}

\newcommand{\op}[2]{O_{#1}(#2)}
\newcommand{\oupper}[2]{O^{#1}(#2)}
\newcommand{\compp}[1]{\operatorname{comp}(#1)}
\newcommand{\comp}[2]{\operatorname{comp}_{#1}(#2)}
\newcommand{\compsol}[1]{\comp{\mathrm{sol}}{#1}}
\newcommand{\layerr}[1]{E(#1)}
\newcommand{\gfitt}[1]{F^{*}(#1)}
\newcommand{\sol}[1]{\operatorname{sol}(#1)}
\newcommand{\zz}[1]{Z(#1)}
\newcommand{\omegaone}[1]{\Omega_{1}(#1)}
\newcommand{\card}[1]{|\,#1\,|}
\newcommand{\syl}[2]{\operatorname{Syl}_{#1}(#2)}
\newcommand{\aut}[1]{\operatorname{Aut}(#1)}
\newcommand{\out}[1]{\operatorname{Out}(#1)}
\newcommand{\gen}[2]{\langle \;#1 \mid #2\; \rangle}
\newcommand{\listgen}[1]{\langle \;#1\; \rangle}
\newcommand{\listset}[1]{\{ \,#1\, \}}

\newcommand{\gf}[1]{\operatorname{GF}(#1)}

\newcommand{\gfpring}[2]{\gf{#1}[#2]}

\newcommand{\ltwo}[1]{\operatorname{L}_{2}(#1)}
\newcommand{\uthree}[1]{\operatorname{U}_{3}(#1)}
\newcommand{\sz}[1]{\operatorname{Sz}(#1)}
\newcommand{\sym}[1]{\operatorname{Sym}(#1)}

\newcommand{\badfourlist}{$\ltwo{2^{r}}$, $\ltwo{3^{r}}$, $\uthree{2^{r}}$ or $\sz{2^{r}}$}

\newcommand{\isomorphic}{\cong}
\newcommand{\normal}{\,\unlhd\,}

\newcommand{\characteristic}{\operatorname{char}}

\newcommand{\br}[1]{\overline{#1}}

\newcommand{\CiteITheoremFourOne}{\cite[Theorem~4.1]{I}}
\newcommand{\CiteISectionAQS}{\cite[\S 6]{I}}
\newcommand{\CiteILemmaTwoTwo}{\cite[Lemma~2.2]{I}}
\newcommand{\CiteILemmaSixSix}{\cite[Lemma~6.6]{I}}
\newcommand{\CiteILemmaSixEightC}{\cite[Lemma~6.8(c)]{I}}

\newcommand{\CiteIISectionPSGPS}{\cite[\S 5]{II}}
\newcommand{\CiteIILemmasSevenSixAndSevenThree}{\cite[Lemmas~7.6 and 7.3]{II}}
\newcommand{\CiteIILemmaFourOne}{\cite[Lemma~4.1]{II}}
\newcommand{\CiteIILemmaFiveFiveB}{\cite[Lemma~5.5(b)]{II}}
\newcommand{\CiteIILemmaFiveFiveC}{\cite[Lemma~5.5(c)]{II}}

\newtheorem{bigEnv}{Big Environment}

\theoremstyle{plain}
\newtheorem{bigTheorem}[bigEnv]{Theorem}
\newtheorem{bigConjecture}[bigEnv]{Conjecture}

\newtheorem{Theorem}{Theorem}[section]

\newtheorem{Lemma}[Theorem]{Lemma}

\newtheorem{Claim}{Claim}

\theoremstyle{definition}

\newtheorem*{UnnumberedDefinition}{Definition}

\theoremstyle{remark}

\numberwithin{equation}{section}

\begin{document}

\title{Primitive pairs of $K$-groups}


%    author one information
\author{Paul Flavell}
\address{The School of Mathematics\\University of Birmingham\\Birmingham B15 2TT\\Great Britain}
\email{P.J.Flavell@bham.ac.uk}
\thanks{A considerable portion of this research was done whilst the author was in receipt
of a Leverhulme Research Project grant and during visits to the Mathematisches Seminar,
Christian-Albrechts-Universit\"{a}t, Kiel, Germany.
The author expresses his thanks to the Leverhulme Trust for their support and
to the Mathematisches Seminar for its hospitality.}

%    \subjclass is required.

\date{}

%    Abstract is required.
\begin{abstract}
    In {\em Primitive pairs of $p$-solvable groups}, J. Algebra 324 (2010) 841--859,
    the author proved a non existence theorem for certain types of amalgams of
    $p$-solvable groups in the presence of operator groups acting coprimely
    on the groups in the amalgam. An application of that work was a new proof
    of the Solvable Signalizer Functor Theorem.
    In this article, the $p$-solvable restriction will be weakened to a $K$-group
    hypothesis.
    An application of this work will be a new proof of the 
    Nonsolvable Signalizer Functor Theorem.
\end{abstract}

\maketitle
\section{Introduction}\label{intro}
The results of \cite{PFpp} will be extended from $p$-solvable groups
to $K$-groups.
Just as an application of \cite{PFpp} is a new proof of the
Solvable Signalizer Functor theorem \cite{PFsft},
an application of this work will be a new proof of
the Nonsolvable Signalizer Functor Theorem.
This work is also a continuation of \cite{I} and \cite{II}
in which the theory of automorphisms of $K$-groups is developed.

We will assume familiarity with \cite{PFpp}.
However, in order to make the statement of the main result
self contained,
we remind the reader of the relevant definitions from \cite{PFpp}.
Throughout, group will mean finite group.

\begin{UnnumberedDefinition}
    Let $M$ be a group and $p$ a prime.
    Then \emph{$M$ has characteristic $p$} if $\cz{M}{\op{p}{M}} \leq \op{p}{M}$.
\end{UnnumberedDefinition}

\begin{UnnumberedDefinition}
    Let $G$ be a group.
    A \emph{weak primitive pair for $G$} is a pair $(M_{1},M_{2})$
    of distinct nontrivial subgroups that satisfy:
    \begin{itemize}
        \item   whenever $\listset{i,j} = \listset{1,2}$ and $1 \not= K \characteristic M_{i}$
                with $K \leq M_{1} \cap M_{2}$ then
                $\n{M_{j}}{K} = M_{1} \cap M_{2}$.
    \end{itemize}
    If $p$ is a prime then the weak primitive pair has \emph{characteristic p}
    if in addition:
    \begin{itemize}
        \item   for each $i$, $M_{i}$ has characteristic $p$
                and $\op{p}{M_{i}} \leq M_{1} \cap M_{2}$.
    \end{itemize}
\end{UnnumberedDefinition}

\begin{UnnumberedDefinition}
    Suppose the group $R$ acts coprimely on the group $G$
    and that $p$ is a prime.
    Then \[
        \op{p}{G;R}
    \]
    is the intersection of all the $R$-invariant Sylow $p$-subgroups of $G$.
\end{UnnumberedDefinition}

\noindent Recall that there do exist $R$-invariant Sylow $p$-subgroups of $G$,
the set of which is permuted transitively by $\cz{G}{R}$.
Moreover, every $R$-invariant $p$-subgroup of $G$ is contained
in an $R$-invariant Sylow $p$-subgroup of $G$.
Consequently, $\op{p}{G;R}$ is the unique maximal $R\cz{G}{R}$-invariant
$p$-subgroup of $G$.

The main result of this paper is the following,
which is a generalization of \cite[Corollary~C]{PFpp}.

\begin{bigTheorem} \label{thmA}
    Let $G$ be a group and $p$ a prime.
    The following configuration is impossible:
    \begin{itemize}
        \item   $(M_{1},M_{2})$ is a weak primitive pair
                of characteristic $p$ for $G$.

        \item   $M_{1}$ and $M_{2}$ are $K$-groups.

        \item   For each $i$ there is an elementary abelian group $R_{i}$
                that acts coprimely on $M_{i}$ and $\op{p}{M_{1};R_{1}} = \op{p}{M_{2};R_{2}}$.
    \end{itemize}
\end{bigTheorem}

\noindent As in \cite{PFpp},
the proof of Theorem~\ref{thmA} uses ideas of Meierfrankenfeld and Stellmacher
\cite{KS,MS1,MS2}.
The Corollaries stated and proved in \cite{PFpp} also remain valid
but with the $p$-solvable hypothesis replaced by a $K$-group hypothesis.

Although Theorem~\ref{thmA} is adequate for the proposed new proof of the
Nonsolvable Signalizer Functor Theorem,
it is almost certain that stronger results hold.
Trivially, the restriction that the groups $R_{i}$ be elementary abelian
is superfluous.
It is included to enable results from \cite{I} and \cite{II} to be used
and hence shorten the presentation.
More significantly,
Theorem~\ref{thmA} would be a consequence of the following:

\begin{bigConjecture}
    Let $p$ be a prime.
    For each $p$-group $P \not= 1$ there exists a characteristic subgroup
    $W(P) \not= 1$ with the property:
    whenever a group $R$ acts coprimely on the group $G$ with
    $G$ of characteristic $p$ and $P = \op{p}{G;R}$ then \[
        W(P) \normal G.
    \]
\end{bigConjecture}

\noindent In the case that $p>3$,
the conjecture is known to be true,
see \cite{PFezj},
the characteristic subgroup in question being
Glauberman's subgroup $K^{\infty}(P)$,
which is defined in \cite{GG}.

\section{Preliminaries} \label{prel}
\begin{Lemma}[See \CiteITheoremFourOne] \label{prel:1}
    Let $r$ be a prime and $K$ a simple $K$-group whose
    order is coprime to $r$.
    \begin{enumerate}
        \item[(a)]  The Sylow $r$-subgroups of $\aut{K}$ are cyclic.

        \item[(b)]  If $\alpha \in \aut{K}$ has order $r$ then
                    $\cz{\aut{K}}{\cz{K}{\alpha}}$ is an $r$-group.
    \end{enumerate}
\end{Lemma}

\begin{Lemma}[See \CiteITheoremFourOne] \label{prel:2}
    Let $R$ be a group of prime order $r$ that acts nontrivially
    and coprimely on the $K$-group $G$.
    Set $K = \gfitt{G}$ and suppose that $K$ is simple.
    Let $p$ be a prime and set $P = \op{p}{G;R}$.
    The following table lists all the cases in which $P \not= 1$.
    \[\begin{array}{c|cccc}
        K               &   \out{K}     & M(K)  & \cz{K}{R}                     &   P \\ \hline
        \ltwo{2^{r}}    &   r           &   1   & 3:2      &   \card{2^{r}+1}_{p} \\
        \ltwo{3{^r}}    &   2\times r  &   2   & 2^{2}:3  &   p=2, P = \op{2}{\cz{K}{R}} \isomorphic 2^{2} \\
        \uthree{2^{r}}  &   3\times 2 \times r  &   3   &  3^{2}:Q_{8}  & p=3, P = \op{3}{\cz{K}{A}} \isomorphic 3^{2} \\
        \sz{2^{r}}      &   r   & 1 (\mbox{$2^{2}$ for $\sz{2^{3}}$})    & 5:4    & \card{2^{r} + 2^{(r-1)/2}\epsilon + 1}_{p}
    \end{array}\]
    where $\epsilon = 1$ if $r \equiv \pm 1 \bmod 8$ and $-1$ if $r \equiv \pm 3 \bmod 8$;
    also an integer on its own stands for the cyclic group of that order.

    Moreover, $P$ is abelian and a Sylow $p$-subgroup of $K$.
\end{Lemma}

\begin{Lemma}\label{prel:3}
    Let $r$ be a prime and $R$ an elementary abelian $r$-group
    that acts coprimely on the quasisimple $K$-group $G$.
    Let $p$ be a prime and suppose that $A$ is an elementary abelian $p$-group
    with $1 \not= A \leq \op{p}{G;R}$.
    Suppose that $V$ is a faithful irreducible $\gfpring{p}{RG}$-module.
    Then \[
        \card{V/\cz{V}{A}} > \card{A}^{2}.
    \]
\end{Lemma}
\begin{proof}
    Set $\br{G} = G/\zz{G}$,
    so $\br{G}$ is simple.
    By Coprime Action,
    $\cz{\br{G}}{R}= \br{\cz{G}{R}}$ so $\br{\op{p}{G;R}} \leq \op{p}{\br{G};R}$.
    Now $V$ is irreducible so $\op{p}{G} = 1$,
    whence $1 \not= \br{A} \leq \op{p}{\br{G};R}$.
    Then $\br{G}$ is one of the groups listed in Lemma~\ref{prel:2}.
    A result of Guralnick and Malle \cite{GM} completes the proof.
    In fact, only Propositions~4.1, 4.3 and 4.6 of \cite{GM} are needed.
\end{proof}

\section{The subgroups $\op{p}{G;R}$} \label{o}
The aim of this section is to prove the following:
\begin{Theorem}\label{o:1}
    Let $r$ be a prime and $R$ an elementary abelian $r$-group.
    Assume that $R$ acts coprimely on the $K$-group $G$
    and that $p$ is a prime.
    Set $P = \op{p}{G;R}$.
    \begin{enumerate}
        \item[(a)]  $P$ acts trivially on $\compsol{G}$.

        \item[(b)]  If $L \in \compsol{G}$ and $P$ acts nontrivially
                    on $L/\sol{L}$ then $L/\sol{L}$ is isomorphic to one of \badfourlist.

        \item[(c)]  $P \leq \op{\mathop{sol},E}{G}$ and $P' \leq \sol{G}$.

        \item[(d)]  Assume $\sol{G} = 1$ and $G = \listgen{P^{G}} \not= 1$.
                    Then $G$ is the direct product of nonabelian simple groups,
                    each of which is isomorphic to one of \badfourlist.
                    Moreover, $P$ is abelian and a Sylow $p$-subgroup of $G$.
    \end{enumerate}
\end{Theorem}

\noindent Recall from \cite{I} that a sol-component of $G$ is a
perfect subnormal subgroup of $G$ that maps onto a
component of $G/\sol{G}$.
An $R$-component of $G$ is the subgroup generated by an orbit
of $R$ on $\compp{G}$.
The sets of sol-components and $R$-components of $G$ are
denoted by $\compsol{G}$ and $\comp{R}{G}$ respectively.
We will use the results of  \CiteISectionAQS\ that describe the
structure of $R$-simple groups and also the results
of \CiteIISectionPSGPS\ to analyze $\op{p}{G;R}$.

\begin{proof}[Proof of Theorem~\ref{o:1}(a)]
    Set $\br{G} = G/\sol{G}$.
    By Coprime Action,
    $\cz{\br{G}}{R} = \br{\cz{G}{R}}$ so $\br{P} \leq \op{p}{\br{G};R}$.
    \CiteILemmaTwoTwo\ implies that the map $\compsol{G} \longrightarrow \compp{\br{G}}$
    defined by $K \mapsto \br{K}$ is a bijection.
    Hence we may assume that $\sol{G} = 1$.
    Apply \CiteIILemmasSevenSixAndSevenThree.
\end{proof}

\begin{Lemma} \label{o:2}
    Assume the hypotheses of Theorem~\ref{o:1}.
    Suppose that $K \in \comp{R}{G}$,
    that $[K,P] \not=1$ and that $\zz{K} = 1$.
    Then:
    \begin{enumerate}
        \item[(a)]  $P \cap K = \op{p}{K;R} \not= 1$.

        \item[(b)]  Let $R_{\infty} = \ker(R \longrightarrow \sym{\compp{K}})$.
                    Then $R_{\infty}$ acts nontrivially on each component of $K$;
                    $\op{p}{K;R} = \op{p}{K;R_{\infty}}$
                    and $\cz{K}{R_{\infty}}$ normalizes $P$.

        \item[(c)]  Each component of $K$ is isomorphic to one of \badfourlist.

        \item[(d)]  $\op{p}{K;R}$ is an abelian Sylow $p$-subgroup of $K$.

        \item[(e)]  $P = \cz{P}{K} \times \op{p}{K;R}$.
    \end{enumerate}
\end{Lemma}
\begin{proof}
    (a). Since $[K,P] \not= 1$, \CiteIILemmaFourOne\ implies that $P \cap K \not= 1$.
    This is in fact a consequence of Lemma~\ref{prel:1}(b).
    Clearly $P \cap K \leq \op{p}{K;R}$.
    Note that $\cz{G}{R}$ acts on $\comp{R}{G}$ by conjugation.
    Let $K_{1},\ldots,K_{n}$ be the distinct $\cz{G}{R}$-conjugates of $K$.
    Set $Q = \op{p}{K_{1};R} \times\cdots\times \op{p}{K_{n};R}$.
    For each $i$, $\cz{G}{R} \cap \n{G}{K_{i}}$ normalizes $\op{p}{K_{i};R}$.
    It follows that $Q$ is $R\cz{G}{R}$-invariant.
    Then $\op{p}{K;R} \leq Q \cap K \leq P \cap K$,
    whence $\op{p}{K;R} = P \cap K$.

    (b). \CiteILemmaSixSix\ implies that the $R\cz{K}{R}$-invariant overdiagonal
    subgroups of $K$ are $R$-simple.
    Since $\op{p}{K;R} \not= 1$ is not $R$-simple,
    it follows that $K$ possesses an $R\cz{K}{R}$-invariant underdiagonal subgroup
    so \CiteILemmaSixEightC\ implies that $R_{\infty}$ acts nontrivially on
    each component of $K$.
    \CiteIILemmaFiveFiveB\ implies $\op{p}{K;R} = \op{p}{K;R_{\infty}}$.
    Let $N_{1},\ldots,N_{m}$ be the components of $K$.
    They are permuted transitively by $R$ and $K = N_{1} \times\cdots\times N_{m}$.
    For each $i$ let $\pi_{i}:K \longrightarrow N_{i}$
    be the projection map.
    Now $\op{p}{K;R_{\infty}} = \op{p}{N_{1};R_{\infty}} \times\cdots\times \op{p}{N_{m};R_{\infty}}$
    so using \CiteIILemmaFiveFiveC\ we obtain \[
        [P,\cz{K}{R_{\infty}}] \leq \op{p}{N_{1};R_{\infty}} \times\cdots\times \op{p}{N_{m};R_{\infty}} %
                                = \op{p}{K;R_{\infty}} \leq P.
    \]
    Thus $\cz{K}{R_{\infty}}$ normalizes $P$.

    (c),(d). We have shown that
    $\op{p}{K;R} = \op{p}{N_{1};R_{\infty}} \times\cdots\times \op{p}{N_{m};R_{\infty}}$.
    Lemma~\ref{prel:2} implies that for all $i$,
    $N_{i}$ is isomorphic to \badfourlist\ and that $\op{p}{N_{i};R_{\infty}}$
    is an abelian Sylow $p$-subgroup of $N_{i}$.
    Thus $\op{p}{K;R}$ is an abelian Sylow $p$-subgroup of $K$.

    (e). For each $i$ set $G_{i} = PN_{i}$.
    Theorem~\ref{o:1}(a) implies that $G_{i}$ is a subgroup.
    Set $\br{G_{i}} = G_{i}/\sol{G_{i}}$.
    Then $\gfitt{\br{G_{i}}} = \br{N_{i}}$ is simple.
    (b) and Coprime Action imply that $\br{P} \leq \op{p}{\br{G_{i}},R_{\infty}}$
    so Lemma~\ref{prel:2} implies that $\br{P} \leq \br{N_{i}}$.
    Then $P \leq \sol{G_{i}}N_{i} \leq \cz{G_{i}}{N_{i}} \times N_{i}$.
    In particular,
    $P$ induces inner automorphisms on $N_{i}$.
    Then $P$ induces inner automorphisms on $N_{1} \times\cdots\times N_{m} = K$
    and so $P \leq \cz{G}{K} \times K$.
    The projection of $P$ into $K$ is $R\cz{K}{R}$-invariant
    and hence is contained in $\op{p}{K;R}$,
    which by (a) is contained in $P$.
    We deduce that $P$ projects onto $\op{p}{K;R}$
    and then that $P = \cz{P}{K} \times \op{p}{K;R}$.
\end{proof}

\begin{proof}[Proof of Theorem~\ref{o:1}(b)]
    Let $\br{G} = G/\sol{G}$.
    As in the proof of (a), $\br{P} \leq \op{p}{\br{G};R}$.
    Note that $P \leq \n{G}{L}$ by (a) and also that $L \cap \sol{G} = \zz{L}$.
    Since $L$ is perfect and $[L,P] \not= 1$ it follows from the
    Three Subgroups Lemma that $[L,P] \not\leq \zz{L}$.
    Then $[\br{L},\br{P}] \not= 1$.
    Hence we may assume that $\sol{G} = 1$.
    Set $K = \listgen{L^{R}} \in \comp{R}{G}$.
    Then $\zz{K} = 1$ and $[K,P] \not= 1$.
    Apply Lemma~\ref{o:2}(c).
\end{proof}

\begin{proof}[Proof of Theorem~\ref{o:1}(d)]
    Since $G = \listgen{P^{G}}$,
    (a) implies that every $R$-component of $G$ is normal in $G$.
    As $\sol{G} = 1$ we may choose $K \in \comp{R}{G}$.
    Note that $\zz{K} \leq \sol{G} = 1$.
    Lemma~\ref{o:2}(e) implies that $P = \cz{P}{K} \times (P \cap K)$,
    then $P \cap K$ is an abelian Sylow $p$-subgroup of $K$
    and that each component of $K$ is isomorphic to \badfourlist.
    Now $K \normal G = \listgen{P^{G}}$
    so $P \leq \cz{G}{K} \times K \normal G$ whence $G = \cz{G}{K} \times K$;
    $P = (P \cap \cz{G}{K}) \times (P \cap K)$;
    $P \cap \cz{G}{K} = \op{p}{\cz{G}{K};R}$
    and $\cz{G}{K} = \listgen{ (P \cap \cz{G}{K})^{\cz{G}{K}} }$.
    Apply induction.
\end{proof}

\begin{proof}[Proof of Theorem~\ref{o:1}(c)]
    As previously,
    we may assume that $\sol{G} = 1$.
    Set $G_{0} = \listgen{P^{G}} \normal G$ and suppose that $G_{0} = G$.
    By induction we have $P \leq \layerr{G_{0}} \normal \layerr{G}$
    and $P' \leq \sol{G_{0}} \leq \sol{G} = 1$.
    Hence we may assume that $\listgen{P^{G}} = G$.
    Apply (d).
\end{proof}

\section{Proof of the main theorem} \label{m}
Throughout this section, $r$ is a prime.
The reader is referred to \cite{PFpp} for the
definition of {\em nearly quadratic $2F$-offender.}

\begin{Theorem}\label{m:1}
    Suppose that $R$ is an elementary abelian $r$-group
    that acts coprimely on the $K$-group $G$,
    that $p$ is a prime and that $V$ is a
    faithful $\gfpring{p}{RG}$-module.
    Then $\op{p}{G;R}$ contains no nearly quadratic $2F$-offender for $RG$ on $V$.
\end{Theorem}
\begin{proof}
    Assume false and let $A \leq \op{p}{G;R}$ be a counterexample.
    We may suppose that $\card{G} + \card{A} + \dim(V) + \card{R}$
    has been minimized.
    A standard reduction,
    see the proof of \cite[Theorem~B]{PFpp},
    shows that $RG$ is irreducible on $V$.
    In particular, \[
        \op{p}{G} = 1.
    \]
    Set $P = \op{p}{G;R}$.
    Now $\listgen{A^{RG}} \normal RG$
    so $\op{p}{\listgen{A^{RG}}} = 1$.
    Using the minimality of $\card{G}$ we obtain
    \begin{equation}\tag{$1$}
        G = \listgen{P^{G}} = \listgen{A^{RG}}.
    \end{equation}
    We claim that if $1 < B < A$ then
    \begin{equation}\tag{$2$}
        \card{\cz{V}{B}/\cz{V}{A}} < \card{A/B}^{2}.
    \end{equation}
    Indeed, since we are studying a counterexample we have
    $\card{V/\cz{V}{A}} \leq \card{A}^{2}$.
    On the other hand,
    the minimality of $\card{A}$ implies $\card{V/\cz{V}{B}} > \card{B}^{2}$.
    Then $(2)$ follows on division.

    Let $N$ be a minimal $R$-invariant normal subgroup of $G$
    chosen minimal subject to $[N,A] \not= 1$.
    Note that $N$ exists because $A \not\leq \op{p}{\zz{G}} = 1$.
    Suppose that $N \not= \oupper{p}{N}$.
    Then $[\oupper{p}{N},A] = 1$ and $(1)$
    implies that $\oupper{p}{N} \leq \zz{G}$.
    In particular, $N$ is nilpotent.
    Then $\op{p}{N} \leq \op{p}{G} = 1$ whence $N = \oupper{p}{N}$,
    a contradiction.
    We deduce that $N = \oupper{p}{N}$.

    We claim that $G = PN$.
    Assume false and set $G_{0} = PN$.
    Let $B = A \cap \op{p}{G_{0}}$.
    Now $[\op{p}{G_{0}},N] \leq \op{p}{N} \leq \op{p}{G} = 1$
    so using the minimality of $\card{G}$ and the fact that $[N,A] \not= 1$,
    we have $1 < B < A$.
    Set \[
        U = \cz{V}{\op{p}{G_{0}}}.
    \]
    Now $N = \oupper{p}{N}$ and $[\op{p}{G_{0}},N] = 1$
    so the \mbox{$P \times Q$}-Lemma implies that $N$ is faithful on $U$.
    As $G_{0} = PN$ it follows that $\cz{G_{0}}{U}$ is a $p$-group
    and then that \[
        \cz{G_{0}}{U} = \op{p}{G_{0}}.
    \]
    The minimality of $\card{G}$ implies that $\card{A/B}^{2} < \card{U/\cz{U}{A}}$.
    Then $(2)$ yields \[
        \card{\cz{V}{B}/\cz{V}{A}} < \card{U/\cz{U}{A}}.
    \]
    However $U \leq \cz{V}{B}$ so $\card{U/\cz{U}{A}} \leq \card{\cz{V}{B}/\cz{V}{A}}$,
    a contradiction.
    We deduce that $G = PN$.

    Suppose that $[\sol{G},A] \not= 1$.
    Then we could have chosen $N$ with $N \leq \sol{G}$,
    whence $G = P\sol{G}$ and $G$ is solvable.
    But then \cite[Theorem~B]{PFpp} supplies a contradiction.
    Thus $[\sol{G},A] = 1$ and then $(1)$ forces \[
        \sol{G} = \zz{G}.
    \]

    Set $\br{G} = G/\zz{G}$.
    Then $\br{P} = \op{p}{\br{G};R}, \sol{\br{G}} = 1$
    and $\br{G} = \listgen{\br{P}^{\br{G}}}$.
    Theorem~\ref{o:1}(d) implies that
    $\br{G}$ is the direct product of simple groups and that $\br{P}$
    is an abelian Sylow $p$-subgroup of $\br{G}$.
    Choosing $N$ to be an $R$-component of $G$ it follows that
    $\br{G}$ is $R$-simple.
    Let $R_{\infty} = \ker( R \longrightarrow \sym{\compp{G}})$.
    Lemma~\ref{o:2}(b) implies that $\br{P} = \op{p}{\br{G};R_{\infty}}$.
    Then as $\br{G} = G/\zz{G}$ we have $P = \op{p}{G;R_{\infty}}$
    and the minimality of $\card{R}$ forces $R = R_{\infty}$.
    In particular, $\br{G}$ is simple.
    We have $\br{G} = \br{G}'$ so $G = \zz{G}G' = \listgen{P^{G}}$.
    As $\zz{G}$ is a $p'$-group we see that $P \leq G'$ whence
    $G = G'$ and $G$ is quasisimple.
    Lemma~\ref{prel:3} implies that $\card{V/\cz{V}{A}} > \card{A}^{2}$,
    contrary to $A$ being a $2F$-offender for $RG$ on $V$.
    The proof is complete.
\end{proof}

\begin{Lemma}\label{m:2}
    Suppose $R$ is an elementary abelian $r$-group that acts coprimely on the $K$-group $G$.
    Suppose that $\gfitt{G} = \op{p}{G}$ for some prime $p$.
    Let $V$ be an elementary abelian normal subgroup of $\op{p}{G;R}$.
    Then either:
    \begin{itemize}
        \item   $V \leq \op{p}{G}$, or

        \item   there exists $A \leq \op{p}{G;R}$ such that $A$
                acts nontrivially and nearly quadratically on $V$ and \[
                    \card{V/\cz{V}{A}} \leq \card{A/\cz{A}{V}}^{2}.
                \]
    \end{itemize}
\end{Lemma}
\begin{proof}
    Set $P = \op{p}{G;R}$ and $G_{0} = \listgen{V^{RG}} \normal G$.
    Then $\gfitt{G_{0}} = \op{p}{G_{0}} \leq \op{p}{G}$
    and $P \cap G_{0} = \op{p}{G_{0};R}$.
    Hence we may suppose that $G = G_{0}$.
    Then \[
        G = \listgen{V^{RG}} = \listgen{P^{G}}.
    \]
    Set \[
        \br{G} = G/\op{p}{G}.
    \]
    \setcounter{Claim}{0}
    \begin{Claim} \label{m:2:Claim1}
        $P \in \syl{p}{G}$ and $\br{V}$ is normal in every Sylow $p$-subgroup
        of $\br{G}$ in which it is contained.
    \end{Claim}
    \begin{proof}
        Set $G_{1} = P\op{p.p'}{G}$.
        Then $\gfitt{G_{1}} = \op{p}{G_{1}}$ and $P \in \syl{p}{G_{1}}$,
        since $\op{p}{G} \leq P$.
        Set $\br{Y} = \op{p'}{\br{G}}$.
        Now $\br{G_{1}} = \br{P}\,\br{Y}$ and $\br{V} \normal \br{P}$.
        In particular, $\cz{\br{Y}}{\br{V}}$ acts transitively
        on the set of Sylow $p$-subgroups of $\br{G_{1}}$ that contain $\br{V}$.
        Consequently, $\br{V}$ is normal in every Sylow $p$-subgroup of
        $\br{G_{1}}$ in which it is contained.
        Hence if $G = G_{1}$ then the claim holds.

        Suppose $G_{1} \not= G$.
        Since $G_{1}$ is $\cz{G}{R}$-invariant
        it follows that $P = \op{p}{G_{1};R}$.
        By induction,
        we may suppose that $V \leq \op{p}{G_{1}}$.
        Then $V$ centralizes $\op{p'}{G/\op{p}{G}}$.
        As $G = \listgen{V^{RG}}$ it follows that
        $G$ centralizes $\sol{G/\op{p}{G}}$,
        whence $\sol{\br{G}} = \zz{\br{G}}$.
        Note that $\zz{\br{G}}$ is a $p'$-group
        and that $\br{P} = \op{p}{\br{G};R}$.

        Put $G^{*} = \br{G}/\sol{\br{G}}$.
        Then $\sol{G^{*}} = 1$ and $G^{*} = \listgen{P^{*G^{*}}}$.
        As $\sol{\br{G}} = \zz{\br{G}}$ it follows from Coprime Action
        that $P^{*} = \op{p}{G^{*};R}$.
        Theorem~\ref{o:1}(d) implies that $P^{*}$
        is an abelian Sylow $p$-subgroup of $G^{*}$.
        Since $\sol{\br{G}}$ is a $p'$-subgroup it follows that $\br{P}$
        is an abelian Sylow $p$-subgroup of $\br{G}$.
        In particular,
        $\br{V}$ is normal in every Sylow $p$-subgroup of $\br{G}$
        in which it is contained.
        As $\op{p}{G} \leq P$ we have $P \in \syl{p}{G}$
        and the claim is established.
    \end{proof}

    Choose $\br{G_{2}}$ minimal subject to \[
        \br{V} \leq \br{G_{2}} \leq \br{G} \quad\mbox{and}\quad \br{V} \not\leq \op{p}{\br{G_{2}}}.
    \]
    Choose $\br{Q} \in \syl{p}{\br{G}}$ with
    $\br{V} \leq \br{Q} \cap \br{G_{2}} \in \syl{p}{\br{G_{2}}}$.
    Claim~\ref{m:2:Claim1} implies $\br{V} \normal \br{P}$
    and $\br{V} \normal \br{Q}$.
    Hence there exists $\br{n} \in \n{\br{G}}{\br{V}}$ with $\br{Q}^{\br{n}} = \br{P}$.
    Replacing $\br{G_{2}}$ with $\br{G_{2}}^{\br{n}}$
    we may suppose that \[
        \br{P} \cap \br{G_{2}} \in \syl{p}{\br{G_{2}}}.
    \]
    Let $G_{2}$ be the full inverse image of $\br{G_{2}}$ in $G$ and let
    $G_{3} = \listgen{V^{G_{2}}} \normal G_{2}$.
    In particular, $V \not\leq \op{p}{G_{3}}$.
    Now $P \cap G_{2} \in \syl{p}{G_{2}}$ whence
    $P \cap G_{3} \in \syl{p}{G_{3}}$.

    \begin{Claim}\label{m:2:Claim2}
        For each $M$ with $V \leq M < G_{3}$, it follows that $V \leq \op{p}{M}$.
    \end{Claim}
    \begin{proof}
        If $\br{M} < \br{G_{2}}$ then this follows from the choice of $G_{2}$,
        so suppose $\br{M} = \br{G_{2}}$.
        Then $G_{2} = \op{p}{G}M$.
        Since $V \leq M$ and $\op{p}{G}$ normalizes $V$ we have
        $\listgen{V^{G_{2}}} \leq M$ so $G_{3} \leq M$,
        a contradiction.
        The claim is established.
    \end{proof}
    Since $V \not\leq \op{p}{G_{3}}$,
    Claim~\ref{m:2:Claim2} and Wielandt's Maximizer Lemma imply that
    $V$ is contained in a unique maximal subgroup $M$ of $G_{3}$.
    Then $V \leq \op{p}{M}$.
    Now $\op{p}{G_{3}} \leq P$ so $\op{p}{G_{3}}$ normalizes $V$.
    Apply \cite[Lemma~3.1]{PFpp} to get the required subgroup $A$.
\end{proof}

Given a group $M$ and a prime $p$,
the subgroup $Y_{M}$ is defined by \[
    Y_{M} = \gen{ \omegaone{\zz{S}} }{ S \in \syl{p}{M} }.
\]
Clearly, $Y_{M} \characteristic M$.
If in addition $M$ has characteristic $p$ then
$Y_{M} \leq \omegaone{\zz{\op{p}{M}}}$ and so $Y_{M}$
may be regarded as a $\gfpring{p}{M}$-module.
A well-known property of $Y_{M}$ is that $\op{p}{M/\cz{M}{Y_{M}}} = 1$,
see \cite{PFpp}.

We are now in a position to prove Theorem~\ref{thmA},
following closely arguments from \cite{PFpp}.
Theorem~\ref{thmA} follows readily from the following
slightly stronger result.

\begin{Theorem}\label{m:3}
    Suppose that $M$ and $S$ are subgroups of the group $G$,
    that $R_{m}$ and $R_{s}$ are elementary abelian groups that act
    coprimely on $M$ and $S$ respectively.
    Assume that:
    \begin{itemize}
        \item   $M$ and $S$ have characteristic $p$ for some prime $p$.

        \item   $\op{p}{M;R_{m}} = \op{p}{S;R_{s}}$.

        \item   $\cz{S}{Y_{M}} \leq M$.

        \item   $M$ and $S$ are $K$-groups.
    \end{itemize}
    Set \[
        P = \op{p}{M}\op{p}{S}.
    \]
    Then the following hold:
    \begin{enumerate}
        \item[(a)]  If $\op{p}{M}$ is abelian then $J(P) = J(\op{p}{M})$.

        \item[(b)]  $J(P) = J(\op{p}{S})$.
    \end{enumerate}
\end{Theorem}
\begin{proof}
    Note that $\op{p}{M} \leq \op{p}{M;R_{m}} = \op{p}{S;R_{s}} \leq S$
    and similarly $\op{p}{S} \leq M$.
    Then $\op{p}{M}$ and $\op{p}{S}$ normalize each other and $P$ is a subgroup.

    \setcounter{Claim}{0}
    \begin{Claim}\label{m:3:Claim1}
        Suppose $V$ is an elementary abelian characteristic $p$-subgroup of $M$
        with $\op{p}{M/\cz{M}{V}} = 1$.
        If $A \leq \op{p}{M;R_{m}}$ acts nontrivially and nearly quadratically
        on $V$ then \[
            \card{A/\cz{A}{V}}^{2} < \card{V/\cz{V}{A}}.
        \]
        Moreover $[V, J(P)] = 1$.
    \end{Claim}
    \begin{proof}
        Set $\br{M} = M/\cz{M}{V}$,
        so $\op{p}{M} = 1$.
        If the inequality is violated then $\br{A}$
        is a nearly quadratic $2F$-offender for $\br{M}$ on $V$.
        But $\br{A} \leq \op{p}{\br{M};R_{m}}$ so Theorem~\ref{m:1}
        supplies a contradiction.

        Suppose that $[V, J(P)] \not= 1$.
        Thompson's Replacement Theorem, see \cite[Theorem~2.4]{PFpp},
        implies there exists $A \in \mathcal A(P)$ with $A$ acting nontrivially
        and quadratically on $V$.
        Since $A \in \mathcal A(P)$ we have
        $\card{V/\cz{V}{A}} \leq \card{A/\cz{A}{V}}$,
        which contradicts the inequality.
    \end{proof}

    \begin{proof}[Proof of (a)]
        Assume that $\op{p}{M}$ is abelian and put $V = \omegaone{\op{p}{M}}$.
        By Coprime Action,
        $\oupper{p}{\cz{M}{V}} \leq \cz{M}{\op{p}{M}} = \op{p}{M}$
        whence $\cz{M}{V} = \op{p}{M}$ and $\op{p}{M/\cz{M}{V}} = 1$.
        By Claim~\ref{m:2:Claim1},
        $J(P) \leq \cz{M}{V} = \op{p}{M} \leq P$
        whence $J(P) = J(\op{p}{M})$.
    \end{proof}

    Put \[
        W = \listgen{ Y_{M}^{\aut{S}} }.
    \]

    \begin{Claim}\label{m:3:Claim2}
        $Y_{M} \normal \op{p}{S}$ and $W$ is elementary abelian.
    \end{Claim}
    \begin{proof}
        Put $V = Y_{M}$.
        Then $\op{p}{M/\cz{M}{V}} = 1$ by \cite[Lemma~2.2]{PFpp}.
        Now $V \normal \op{p}{M;R_{m}} = \op{p}{S;R_{s}}$.
        Lemma~\ref{m:2},
        with $S$ in the role of $G$,
        and Claim~\ref{m:3:Claim1} imply $V \leq \op{p}{S}$.
        Since $\op{p}{S} \leq \op{p}{S;R_{s}}$,
        we have $V \normal \op{p}{S}$.

        Suppose that $W$ is nonabelian.
        Then there exists $a \in \aut{S}$ with $[V,V^{a}] \not= 1$.
        Also $[V,V^{a},V^{a}] \leq [V \cap V^{a},V^{a}] = 1$.
        In particular,
        $V^{a}$ acts nontrivially and quadratically on $V$.
        Claim~\ref{m:3:Claim1} implies that \[
            \card{V^{a}/\cz{V^{a}}{V}} < \card{V/\cz{V}{V^{a}}}.
        \]
        Moreover, $[V^{a^{-1}},V] \not= 1$ whence \[
            \card{V^{a^{-1}}/\cz{V^{a^{-1}}}{V}} < \card{V/\cz{V}{V^{a^{-1}}}}.
        \]
        Conjugating the second inequality by $a$ contradicts the first.
    \end{proof}

    Let \[
        Q = \op{p}{S \bmod \cz{S}{W}}.
    \]

    \begin{Claim}\label{m:3:Claim3}
        $P \cap Q \leq \op{p}{S}$.
    \end{Claim}
    \begin{proof}
        Choose $T$ with $P \cap Q \leq T \in \syl{p}{Q}$.
        Then $S = Q\n{S}{T} = \cz{S}{W}\n{S}{T} = (M \cap S)\n{S}{T}$
        because $\cz{S}{W} \leq \cz{S}{Y_{M}} \leq M$.
        Now $P \cap Q$ is normalized by $M \cap S$ because
        $P = \op{p}{M}\op{p}{S} \normal M \cap S$,
        whence $\listgen{ (P \cap Q)^{S} } = \listgen{ (P\cap Q)^{\n{S}{T}} } \leq T$
        and so $P \cap Q \leq \op{p}{S}$.
    \end{proof}

    \begin{Claim}\label{m:3:Claim4}
        $[W,J(P)] = 1$.
    \end{Claim}
    \begin{proof}
        Assume false.
        By the Replacement Theorem,
        there exists $A \in \mathcal A(P)$ such that $A$
        acts nontrivially and quadratically on $W$.

        Suppose that $A \leq Q$.
        Claim~\ref{m:3:Claim3} implies that $A \leq \op{p}{S}$
        whence $A \leq J(\op{p}{S})$.
        Since $A \in \mathcal A(P)$ it follows that $J(\op{p}{S}) \leq J(P)$.
        By Claim~\ref{m:3:Claim1},
        $[Y_{M},J(P)] = 1$ whence $[Y_{M},J(\op{p}{S})] = 1$.
        As $W = \listgen{ Y_{M}^{\aut{S}} }$ we obtain $[W, J(\op{p}{S})] = 1$
        and then $[W,A] = 1$,
        a contradiction.
        We deduce that $A \not\leq Q$.

        Set \[
            A_{0} = A \cap Q.
        \]
        By Claim~\ref{m:3:Claim3},
        $A_{0} \leq \op{p}{S}$ and $\cz{A}{W} \leq \op{p}{S} \leq Q$ so \[
            \cz{A}{W} = \cz{A_{0}}{W}.
        \]
        Let $q = q(Y_{M},\op{p}{S})$.
        The reader is referred to \cite{PFpp} or \cite{KS} for the definition of this parameter.
        Since $\op{p}{S} \leq \op{p}{S;R_{s}} = \op{p}{M;R_{m}}$,
        Claim~\ref{m:3:Claim1} implies \[
            q > 2.
        \]
        Now $[Y_{M},A_{0},A_{0}] = 1$ because $A$ is quadratic on $W$.
        Applying \cite[Lemma~2.5]{PFpp},
        we obtain
        \begin{equation}\tag{$1$}
            \card{A_{0}/\cz{A_{0}}{W}}^{q} \leq \card{W/\cz{W}{A_{0}}}.
        \end{equation}
        As $A \in \mathcal A(P)$ we have $\card{A} \geq \card{W\cz{A}{W}}$ so
        \begin{equation}\tag{$2$}
            \card{A/\cz{A}{W}} \geq \card{W/\cz{W}{A}}.
        \end{equation}
        Raising $(2)$ to the power $q$,
        dividing by $(1)$ and using $\cz{A}{W} = \cz{A_{0}}{W}$ yields
        \begin{align*}
            \card{A/A_{0}}^{q} &\geq \card{W/\cz{W}{A}}^{q-1}\card{\cz{W}{A_{0}}/\cz{W}{A}} \\
                                &\geq \card{W/\cz{W}{A}}^{q-1}
        \end{align*}
        because $A_{0} \leq A$.
        Then \[
            \card{A/A_{0}}^{1+1/(q-1)} \geq \card{W/\cz{W}{A}}.
        \]
        Since $q>2$ and $A_{0} < A$ we obtain \[
            \card{A/A_{0}}^{2} > \card{W/\cz{W}{A}}.
        \]
        Thus $A/\cz{A}{W}$ is a nearly quadratic $2F$-offender
        for $S/\cz{S}{W}$ on $W$.
        Since $A \leq \op{p}{S;R_{s}}$,
        Theorem~\ref{m:1} supplies a contradiction.
        Hence $[W,J(P)] = 1$.
    \end{proof}

    Claims~\ref{m:3:Claim3} and \ref{m:3:Claim4} imply $J(P) \leq \op{p}{S}$.
    Since $\op{p}{S} \leq P$ it follows that $J(P) = J(\op{p}{S})$
    and the proof is complete.
\end{proof}

\begin{proof}[Proof of Theorem~\ref{thmA}]
    Assume false.
    Two applications of the previous theorem yield
    $J(\op{p}{M_{1}}) = J(\op{p}{M_{1}}\op{p}{M_{2}}) = J(\op{p}{M_{2}})$.
    Since $J(\op{p}{M_{i}}) \normal M_{i}$,
    this gives a contradiction since no nontrivial subgroup of
    $M_{1} \cap M_{2}$ can be normal in both $M_{1}$ and $M_{2}$
    by the definition of a weak primitive pair.
\end{proof}

\bibliographystyle{amsplain}

\end{document}